\documentclass[11pt,leqno]{article}

\usepackage{amssymb,amsmath,amsthm,mysects,url,rotating}
 \usepackage{graphicx,epsfig}
 \usepackage{xcolor,graphicx}
\newcommand{\vp}{\varepsilon}
\newcommand{\bb}[1]{\mathbb{#1}}
\newcommand{\cl}[1]{\mathcal{#1}}

\newcommand{\ovl}{\overline}

\theoremstyle{plain}
\newtheorem{thm}{Theorem}[section]

\newtheorem{cor}[thm]{Corollary}

\theoremstyle{definition}

\theoremstyle{remark}
\newtheorem{rem}[thm]{Remark}

\numberwithin{equation}{section}

\def\E{\bb E}

\def\P{\bb P}
\begin{document}
\def\d{\delta}

\title{Tripartite Bell inequality, random matrices and trilinear forms}  
\author{Gilles Pisier}
\maketitle

\begin{abstract} In this seminar report,
we present in detail the proof of a recent result 
due to J. Bri\"et and T. Vidick, improving an estimate in a 2008 paper by
D. P\'erez-Garc\'{\i}a, M. Wolf, C. Palazuelos, I. Villanueva, and M. Junge, 
estimating the growth of  the deviation in the tripartite Bell inequality. The proof requires a delicate estimate of the norms
of   certain trilinear (or $d$-linear) forms on Hilbert space with coefficients
in the second Gaussian Wiener chaos. Let $E^n_{\vee}$ (resp. $E^n_{\min}$)  denote
$ \ell_1^n \otimes \ell_1^n\otimes \ell_1^n$ equipped with
the injective (resp. minimal) tensor norm. Here $ \ell_1^n$ is equipped
with its maximal operator space structure. The Bri\"et-Vidick method
yields that the identity map $I_n$ satisfies (for some $c>0$) $\|I_n:\  E^n_{\vee}\to E^n_{\min}\|\ge c n^{1/4} (\log n)^{-3/2}.$ Let $S^n_2$ denote
the (Hilbert) space of $n\times n$-matrices equipped with the Hilbert-Schmidt norm.
While a lower bound closer to $n^{1/2} $  is still open, their method
produces an interesting, asymptotically almost sharp, related estimate for the 
 map $J_n:\ S^n_2\stackrel{\vee}{\otimes} S^n_2\stackrel{\vee}{\otimes}S^n_2 \to \ell_2^{n^3} \stackrel{\vee}{\otimes} \ell_2^{n^3} $ taking
 $e_{i,j}\otimes e_{k,l}\otimes e_{m,n}$ to $e_{[i,k,m],[j,l,n]}$. 

\end{abstract}

\section{Tripartite Bell inequality}
We will prove the following theorem
due to J. Bri\"et and T. Vidick, improving an estimate in Junge etal. The proof
in \cite{BV} was kindly  explained to me in detail by T. Vidick. The improvements below (improving the power of the logarithmic term)
are routine refinements of the ideas in \cite{BV}.

\begin{thm}\label{thm1}  Let $Y^{(N)}$ be $N\times N$ a  Gaussian random matrix
with Gaussian entries all i.i.d. of mean zero and $L_2$-norm equal
to $N^{-1/2}$. Let $Y_j^{(N)}$ ($j=1,2,\cdots$) be a sequence of independent copies
of $Y^{(N)}$.
There is a constant C
such that with large probability we have
for all scalars $a_{i j}$ with $1\le i,j\le N$

$$\|\sum_{i,i'=1}^N a_{i i'} Y_i^{(N)}\otimes Y_{i'}^{(N)} \|\le C  (\log N)^{3/2}(\sum_{i,i'=1}^N |a_{i i'} |^2)^{1/2}$$

Equivalently,  let $g_{i,k,l}$   be i.i.d. Gaussian normal random variables, with 
${i,k,l}\le N$. Let  $g'_{i,k,l}$ be an independent copy of the family  $g_{i,k,l}$.
Then the norm $\|{\cl T}\|_{\vee}$  of the tensor $${\cl T}=\sum g_{i,k,l} \ g'_{i',k',l'} e_{i i'}\otimes e_{k k'} \otimes e_{l l'}$$
in the triple injective tensor product  $\ell^{N^2}_2 \stackrel{\vee}{\otimes} \ell^{N^2}_2 \stackrel{\vee}{\otimes} \ell^{N^2}_2$ satisfies for some $C$
$$\E \|{\cl T}\|_{\vee} \le  C (\log N)^{3/2} N $$

\end{thm}
\begin{rem} Note that if we replace in ${\cl T}$ the random coordinates by
a family of  i.i.d. Gaussian normal variables indexed by $N^6$, then by well known estimates
(e.g. the Chevet inequality)   the corresponding random tensor, denoted by  $\hat {\cl T}$, satisfies
$\E \|\hat {\cl T}\|_{\vee} \le  C   N $.
\end{rem}

\begin{rem} 
Let $g$ be a Gaussian vector   in a finite dimensional (real) Hilbert space $H$ and let $g'$ be an independent copy of $g$.
We assume (for simplicity) that  the distribution of $g$ and  $g'$ is the canonical Gaussian measure on
$H$. 
Let $u_i\ (i=1,\cdots,M)$ be operators on $H$.
Let $Z_i=\langle u_i g , u_i g' \rangle$, and let
$\hat Z_i=\langle u_i g , u_i g\rangle - \E\langle u_i g, u_i g \rangle .$  We have then for any $p\ge 1$
\begin{equation}\label{1}  2^{-1}\| \sup_{i\le M}|\hat Z_i|  \|_p\le \| \sup_{i\le M}|Z_i| \|_p\le  \ \| \sup_{i\le M}|\hat Z_i| \|_p,
\end{equation}
and hence
\begin{equation}\label{2} \| \sup_{i\le M}\langle u_i g, u_i g  \rangle  \|_p\le  2 \| \sup_{i\le M}|Z_i| \|_p +
  \sup_{i\le M} \|u_i\|^2_{S_2} ,
\end{equation}
wher $\| \ \|_{S_2}$ denotes the Hilbert-Schmidt norm.
Indeed, (denoting $\approx$ equality in distribution)
we have
$$(g,g') \approx (2^{-1/2}(g+g'),  2^{-1/2}(g-g') ) .$$
Therefore 
$$\langle u_i g, u_i g\rangle  -\langle u_i g', u_i g'\rangle  \approx
 2 \langle u_i g , u_i g'\rangle.$$
Thus, if $\{\hat Z'_i\}$ is an  independent copy of $\{\hat Z_i\}$, we have
$$ \hat Z_i-\hat Z'_i \approx  2 Z_i.$$
From this \eqref{1} follows easily and \eqref{2} is an immediate consequence.
\end{rem}

\begin{rem}
\def\rk{ {\rm rk}}
Let us denote by $S_{2,1}(H)$ the class of operators $u$ on $H$ such that
the eigenvalues $\lambda_j$  of $ |u|$ (rearranged  as usual with multiplicity in non-increasing order)
satisfy 
$$\sum j^{-1/2} \lambda_j<\infty,$$
equipped with the quasi-norm (equivalent to a norm)
$$\|u\|_{2,1}=\sum j^{-1/2} \lambda_j<\infty.$$
It is clear that by Cauchy-Schwarz (for some constant $C$)
\begin{equation}\label{3} \|u\|_{2,1} \le C (\log \rk(u))^{1/2} \|u\|_{2}.\end{equation}

It is well known that the unit ball of $S_{2,1}(H)$ is equivalent (up to absolute constants)
to the closed convex hull of the set formed by all 
operators of the form $u=k^{-1/2} P$ where $P$ is a (orthogonal) projection of rank $k$. In particular,
for any $u$ we have
(for some constant $C>0$)
$$C^{-1}\|u\|_2\le \|u\|_{2,1}.$$
 
 Therefore, if $Z$ is a trilinear form on $S_{2,1}(H)$, 
 we have (for some constant $C$)
 \begin{equation}\label{4}\sup\{ |Z(r,s,t)|\mid \|r\|_{2,1}\le 1,\|s \|_{2,1}\le 1,\|t\|_{2,1}\le 1\}  \le
 C\sup_{k,l,m} (k l m)^{-1/2} \sup   |Z(R,S,T)|\ \end{equation}
 where the second supremum runs over all integers ${k,l,m}$
 and the third one over all projections $R,S,T$ of  rank respectively ${k,l,m}$.
 Let us denote by ${\cl P}(k)$ the set of projections of rank $k$.
By \eqref{3} letting $d=\dim(H)$, this implies (for some constant $C$)
 \begin{equation}\label{5}\sup_ {\|x\|_{2}\le 1,\|y \|_{2}\le 1,\|z\|_{2}  \le 1}  |Z(x,y,z)| \le C
 (\log d)^{3/2}
 \sup_{k,l,m}  (k l m)^{-1/2}  \sup_{(R,S,T)\in {\cl P}(k)\times {\cl P}(l)\times {\cl P}(m)}   |Z(R,S,T)|.
  \end{equation}
\end{rem}

\begin{proof}[Proof of Theorem]
We identify $\ell^{N^2}_2$ with the Hilbert-Schmidt class $S^{N^2}_2$.
Then, viewing $R$ as an operator (or matrix)  acting on $\ell^{N}_2$,   we denote by $\|R\|_2$ and $\|R\|_\infty$ respectively its
Hilbert-Schmidt  norm and   operator norm.
Let
$$Z(R,S,T)=\sum g_{i,k,l} \ g'_{i',k',l'}  R_{i i'} S_{k k'}  T_{l l'} $$
The norm of ${\cl T}$ is the supremum of $Z( R,S,T)$
over $R,S,T$ in the unit ball of $\ell^{N^2}_2$. 

Note that $Z(R,S,T)=\langle  (R \otimes S\otimes T )g, g' \rangle$
where $g,g'$ are independent canonical random vectors on $\ell_2^{N^3}$,
and $u= R \otimes S\otimes T$ is an operator on $\ell_2^{N^3}=\ell^{N}_2\otimes \ell^{N}_2\otimes \ell^{N}_2$.
 
 Fix integers $r,s,t$ and $\d>0$.
 Let ${\cl P}(r,s,t)={\cl P}(r)\otimes {\cl P}(s)\otimes {\cl P}(t)$. 
 Let ${\cl P}_\d(r)$ be a $\d$-net in ${\cl P} (r)$ with respect to the norm in $S_2$.
 It is easy to check that we can find
 such a net with at most $\exp\{c(\d)rN\}$ elements, so we may assume
 that $|{\cl P}_\d(r))|\le \exp\{c(\d)rN\}$.
 Let $${\cl P}_\d(r,s,t)={\cl P}_\d(r)\otimes {\cl P}_\d(s)\otimes {\cl P}_\d(t).$$
 Note that ${\cl P}_\d(r,s,t)$ is a $3\d$-net in ${\cl P}(r,s,t)$ and $|{\cl P}_\d(r,s,t)|\le \exp\{ 3 c(\d)(r+s+t)N\}$. 
 
 Let $$\|Z\|=\sup_ {\|x\|_{2}\le 1,\|y \|_{2}\le 1,\|z\|_{2}  \le 1}  |Z(x,y,z)| ,$$
 and
 $$\|Z\|_{\bullet}=\sup_ {\|x\|_{2,1}\le 1,\|y \|_{2,1}\le 1,\|z\|_{2,1}  \le 1}  |Z(x,y,z)| .$$
 Recall that by \eqref{5} $$\|Z\|_{\bullet} \le C ( \log N)^{3/2}\|Z\|.$$
 
\noindent {\bf Claim:} We claim that for some constant $C_\d$ we have
$$\| \sup_{(R,S,T)\in {\cl P}_\d(r,s,t)} (rst)^{-1/2}|Z(R,S,T)|  \|_N \le C_\d N .$$

We will use a bound
of Lata\l a (actually easy to prove in the bilinear case)
that says that for some absolute constant $c$        we have for all $p\ge 1$
$$\|Z(R,S,T)\|_p \le  c( p^{1/2} \|R\|_2\|S\|_2\|T\|_2 + p \|R\|_{\infty}\|S\|_{\infty}\|T\|_{\infty}).$$ 

\def\d{\delta}

Let us record here for further reference the obvious inequality

 \begin{equation}\label{7}
   \|\sup_{i\le M} |Z(R_i,S_i,T_i)|\|_q \le  M^{1/q} \sup_{i\le M} \|Z(R_i,S_i,T_i)|\|_q.
\end{equation}

When we take the sup over a family
$R_i,S_i,T_i$ indexed by $i=1,...,M$
in the unit ball of $S_2^N$ and such that for all $i$
we have $\|R_i\|_\infty \le r^{-1/2}$,  $\|S_i\|_\infty \le s^{-1/2}$, $\|T_i\|_\infty \le t^{-1/2}$,
this gives us
$$\|\sup_{i\le M} |Z(R_i,S_i,T_i)| \|_p \le c M^{1/p} (p^{1/2} +p (rst)^{-1/2}).$$
Choosing $p=\log M$  and $p\ge q$ we find a fortiori
 \begin{equation}\label{6}
   \|\sup_{i\le M} |Z(R_i,S_i,T_i)|\|_q \le c'   ((\log M)^{1/2} +(\log M) (rst)^{-1/2}).
\end{equation}

To prove the claim we may reduce to triples $(r,s,t)$
such that $t=\max(r,s,t)$. Indeed exchanging the roles of $(r,s,t)$,
we treat similarly the cases $r=\max(r,s,t)$ and $s=\max(r,s,t)$ and the desired result
follows with a tripled constant $C_\d$.

We will treat separately the sets\\  \centerline{ $A=\{(r,s,t)\mid rs>t, \ t=\max\{r,s\}\}$ and
$B=\{(r,s,t)\mid rs  \le t,  \ t=\max\{r,s\}\}$.}\\ Note that both sets have at most $N^3$ elements.

$\bullet$ 
Fix $(r,s,t)\in A$. By \eqref{6} with $q=N$, we have
$$\| \sup_{(R,S,T)\in {\cl P}_\d(r,s,t)} (rst)^{-1/2}|Z(R,S,T)|\|_N \le 
 c'   ((\log |{\cl P}_\d(r,s,t)|)^{1/2} +(\log |{\cl P}_\d(r,s,t)|) (rst)^{-1/2}),
$$
and hence
 \begin{equation}\label{8}\le c'  (3 c(\d)(r+s+t)N)^{1/2} + c'  (3 c(\d)(r+s+t)N)(rst)^{-1/2}.
\end{equation}
Now on the one hand $(r+s+t)N)^{1/2} \le 3^{1/2} N$ and on the other hand,
since we assume $ t =\max\{r,s\}$ and $rs>t $, we have $(r+s+t)N)(rst)^{-1/2}\le 3tN(rst)^{-1/2}\le 3N$.
So assuming $(r,s,t)\in A$ we obtain
$$\| \sup_{(R,S,T)\in {\cl P}_\d(r,s,t)} (rst)^{-1/2}|Z(R,S,T)| \|_N\le  3 c'  c(\d) (3^{1/2}+3) N .$$

$\bullet$ Now assume $(r,s,t)\in B$ and in particular $rs  \le  t$. 
 By Cauchy-Schwarz if $R,S,T$ are projections
 we have
 $|Z(R,S,T)| =\langle g, (R \otimes S  \otimes T )g'\rangle\le \langle g, (R \otimes S  \otimes T )g\rangle^{1/2} \langle g', (R \otimes S  \otimes T )g'\rangle^{1/2} .$
  We may write a fortiori  
 $$\|Z(R,S,T)|\|_p \le  |\langle g, (R \otimes S  \otimes T )g  \rangle  \|_p .$$
 Therefore by \eqref{2} and since $T\le I$ (and hence $ R \otimes S  \otimes T\le R \otimes S  \otimes I$)  we have
 $$\|Z(R,S,T) \|_p \le 2 \|Z(R,S,I) \|_p + rsN .$$
 Similarly we find
  \begin{equation}\label{9} \|\sup_{(R,S,T)\in {\cl P}_\d(r,s,t)}  |Z(R,S,T)| \|_p \le 2 \| \sup_{(R,S,I)\in {\cl P}_\d(r,s,N)} Z(R,S,I)|\|_p + rsN ,\end{equation}
 and hence by \eqref{6} again (we argue as for \eqref{8} above with $q=N$, but note however that in the present case   $T=I$ is fixed so
 the supremum over ${(R,S,I)\in {\cl P}_\d(r,s,N)}$  runs over at most $\exp\{ 3 c(\d)(r+s)N\}$ elements and we may use
  $p=3 c(\d)(r+s)N$) we find
  $$\|\sup_{(R,S,I)\in {\cl P}_\d(r,s,N)} (rst)^{-1/2}|Z(R,S,I)| \|_N \le
   c'  (3 c(\d)(r+s)N)^{1/2} (N/t)^{1/2}+ c'  (3 c(\d)(r+s)N)(rst)^{-1/2}.$$
 But now,  since we assume $rs  \le  t$, $r+s\le 2 rs \le 2t$  so that
 $(r+s)(rst)^{-1/2}\le 2 (rs/t)^{1/2}\le 2$, and hence we find 
  $\|\sup_{(R,S,I)\in {\cl P}_\d(r,s,N)} (rst)^{-1/2}|Z(R,S,I)| \|_N \le c_3(\d) N$.
  Substituting this in \eqref{9} yields
  $$\|\sup_{(R,S,T)\in {\cl P}_\d(r,s,t)}   (rst)^{-1/2}|Z(R,S,T)| \|_N \le 2c_3(\d) N +  (rs/t)^{1/2}N\le   (2c_3(\d) +1)N.$$
This completes the proof of the above claim.

Using the   claim, we conclude the proof as follows: To
pass from ${\cl P}_\d(r,s,t)$   to
 ${\cl P}(r,s,t)$  we first note that if (say) $P,P'$ are both projections of rank $r$, 
  by \eqref{3}  $\|P-P'\|_2 \le \d$ implies
 (for some constant $c_1$)  that 
 $r^{-1/2}\|P-P'\|_{2,1} \le c_1 \d$. Thus, using this for $r,s,t$ successively, we find 
 $$ \sup_{(R,S,T)\in {\cl P}(r,s,t)} (rst)^{-1/2}|Z(R,S,T)| \le  \sup_{(R,S,T)\in {\cl P}_\d(r,s,t)} (rst)^{-1/2}|Z(R,S,T)| +3 c_1\d  \|Z\|_{\bullet}.$$
 This implies
  $$\| \sup_{(R,S,T)\in {\cl P}(r,s,t)} (rst)^{-1/2}|Z(R,S,T)| \|_N \le  \|\sup_{(R,S,T)\in {\cl P}_\d(r,s,t)} (rst)^{-1/2}|Z(R,S,T)| \|_N +3 c_1 \d  \| \|Z\|_{\bullet}\|_N.$$
  Using \eqref{7} to estimate the sup over the $N^3$ intergers $r,s,t$ we find
  (recalling \eqref{4})
  $$   \| \|Z\|_{\bullet}\|_N
  \le CN^{3/N} \sup_{r,s,t}( \|\sup_{(R,S,T)\in {\cl P}_\d(r,s,t)} (rst)^{-1/2}|Z(R,S,T)| \|_N +3c_1 \d  \| \|Z\|_{\bullet}\|_N),$$
and hence by the claim
  $$ \le  8C C_\d N + 24 C c_1\d  \| \|Z\|_{\bullet}\|_N $$
   from which  follows that
  $$   \| \|Z\|_{\bullet}\|_N
  \le (1-24 c_1 C\d )^{-1} 8C C_\d N  .$$
  Observe that $  \E \|Z\| \le (\log N)^{3/2} \E  \|Z\|_{\bullet} \le   (\log N)^{3/2}  \| \|Z\|_{\bullet}\|_N$.
 Thus, if $\d$ is small enough,   chosen so that   say $24 c_1 C\d=1/2$,
 we finally obtain a fortiori
$$ \E \|Z\| \le 16 C C_\d   (\log N)^{3/2} N  .$$
Actually since we obtain the same bound for  $(\E \|Z\|^N)^{1/N}$
we also obtain for suitable positive constants $c_2,c_3$ that
$$\P \{\|Z\| > c_3 N (\log N)^{3/2} \}\le \exp{-c_2 N}.$$

\end{proof}

The Theorem has the following application improving a result in
Junge etal \cite{Jetal}:

\begin{thm}\label{thm2}  Consider the following two norms
for an element
$t=\sum\nolimits_{ijk} t_{ijk} e_i\otimes  e_j \otimes e_k$
in the triple tensor product $\ell_1^n \otimes \ell_1^n\otimes \ell_1^n$:
\begin{equation}\label{10}\|t\|_{\min}=\sup\{ \| \sum\nolimits_{ijk} t_{ijk} u_i\otimes  v_j \otimes w_k\|_{B(H\otimes_2 H  \otimes_2 H)  }\}\end{equation}
where the sup runs over all possible Hilbert spaces $H$ and all possible unitary operators
$u_i,v_j,w_k$ acting on $H$, and also:
\begin{equation}\label{11}\|t\|_{\vee}=\sup\{ |\sum\nolimits_{ijk} t_{ijk} x_i   y_j   z_k|\}\end{equation}
where the sup runs over all unimodular scalars $x_i,y_j, z_k$ or equivalently
the sup is as before but restricted to $\dim(H)=1$.
Let 
\begin{equation}\label{12}C_3(n)= \sup\{ \|t\|_{\min} \mid \|t\|_{\vee}\le 1\}.\end{equation}
Then we have for some constant $C'>0$ (independent of $n$)
$$C_3(n)\ge C' n^{1/4} (\log n)^{-3/2}.$$

\end{thm}
 \begin{rem} It is well known that the supremum in \eqref{10} is unchanged if we restrict
the supremum to finite dimensional spaces $H$.
Moreover, we have also
$$   \|t\|_{\min}=\sup\{ \| \sum\nolimits_{ijk} t_{ijk} u_i   v_j   w_k\|_{M_N  }\},$$
where the supremum runs over all $N$
and all $N\times N$-unitary matrices $u_i,v_j,w_k$
such that $u_iv_j=v_ju_i$, $u_iw_k=w_ku_i$
and $w_kv_j=v_jw_k$ for all ${i,j,k}$.\\
Note that \eqref{11} corresponds again to restricting this sup to $N=1$.

 \end{rem}
 \begin{proof}[Proof of Theorem \ref{thm2}] Let $n=N^2$.  We again identify ${\ell^{N^2}_2}$
 with the space of $N \times N$ matrices equipped with the HS norm.
 Let $\{ u_j\mid j\le N^2\}$ be
 an orthogonal  basis in ${\ell^{N^2}_2}$ 
 consisting of unitaries (this is   called an EPR basis in quantum information).
 Note that $\|u_j\|_2=\sqrt{N}$ for all $j$.
 Then 
 $u_i\otimes  u_j \otimes u_k$ $(i,j,k \le N^2 )$ forms an 
 orthogonal  basis in ${\ell^{N^2}_2}\otimes {\ell^{N^2}_2}\otimes {\ell^{N^2}_2}$.
 
Consider now $T \in {\ell^{N^2}_2}\otimes {\ell^{N^2}_2}\otimes {\ell^{N^2}_2}$
and let 
$$T=\sum \hat T_{ijk} u_i\otimes  u_j \otimes u_k$$
be its development on that orthogonal basis, so that
$\hat T_{ijk}=N^{-3}\langle T, u_i\otimes  u_j \otimes u_k\rangle$.
Consider now
$t=\sum\nolimits_{ijk} \hat T_{ijk}\  e_i\otimes  e_j \otimes e_k$.
 Then
for any unimodular scalars $x_i,y_j, z_k$ we have \\
$$\sum \hat T_{ijk} x_i\otimes  y_j \otimes z_k=  \langle T, X\otimes  Y\otimes Z\rangle$$
with $X=\sum x_i u_i$, $Y=\sum y_j u_j$, $Z=\sum z_k u_k$
and hence
$$\|t\|_{\vee}\le  N^{3/2}\sup\{ |\langle T, X\otimes  Y\otimes Z\rangle| \mid X,Y,Z \in B_{\ell^{N^2}_2}\}.$$ 
But since we also have
$T= \sum \hat T_{ijk} u_i\otimes  u_j \otimes u_k$, we have
$$\|t\|_{\min}\ge   \|T\|_{B({\ell^{N}_2}\otimes_2 {\ell^{N}_2}  \otimes_2 {\ell^{N}_2})  }.$$
Now consider as before
$${ T}=\sum g_{i,k,l} \ g'_{i',k',l'} e_{i i'}\otimes e_{k k'} \otimes e_{l l'}.$$
With this choice of ${ T}$  by the preceding Theorem we find with large probability
$\sup\{ |\langle T, X\otimes  Y\otimes Z\rangle| \mid X,Y,Z \in B_{\ell^{N^2}_2}\} \le C N (\log N)^{3/2} ,$
and hence
$$\|t\|_{\vee}\le  C N^{5/2}(\log N)^{3/2},$$
but also (since $T$ is a rank one operator)
$\|T\|_{B({\ell^{N}_2}\otimes_2 {\ell^{N}_2}  \otimes_2 {\ell^{N}_2})  }=(\sum |g_{i,k,l}|^2)^{1/2}  (\sum |g'_{i,k,l}|^2)^{1/2}$ and the latter is  concentrated around its mean and hence with large probability
$\ge  N^3/2$.
Thus we conclude that
$$C(N^2)\ge (2C)^{-1}N^{1/2} (\log N)^{-3/2}. $$
 \end{proof}
 \begin{rem}\label{rem7} The same method works in the $d$-linear case.
    Consider
$$T=\sum g_{i(1),...,i(d)}   g'_{i'(1),...,i'(d)}   e_{i(1)i'(1)} \otimes\cdots  \otimes e_{i(d)i'(d)}\in \ell^{N(1)^2}_1 \otimes\cdots  \otimes \ell^{N(d)^2}_1.$$
Let $\{u_{i}^{(j)}\mid 1\le i\le N(j)^2\}$ be an orthogonal basis in $S^{N(j)}_2$ formed of unitary matrices.
We will denote by ${\underline{ i}}$ the elements of the set $\underline{ I}=[N(1)^2\times \cdots \times N(d)^2]$.
Let
$$T=\sum \hat T(\underline{ i}) u_{\underline{ i}}$$
be its orthogonal development according to the basis
formed by
$$\forall \underline{ i} =(\underline{ i}(1),\cdots,\underline{ i}(d)\in \underline{ I}\qquad  u_{\underline{ i}} =u_{\underline{ i}(1)}^{(1)}\otimes \cdots \otimes  u_{\underline{ i}(d)}^{(d)},$$
so that
$$\hat T(\underline{ i})={(N(1)\cdots N(d))}^{-1} \langle T ,  u_{\underline{ i}}\rangle.$$   

Let us denote also by   $e_{\underline{ i}}=e_{\underline{ i}(1)} \otimes \cdots \otimes  e_{\underline{ i}(d)}$ the canonical basis in $\ell^{N(1)^2}_1 \otimes\cdots  \otimes \ell^{N(d)^2}_1$.
Let then 
$$t=\sum \hat T(\underline{ i}) e_{\underline{ i}}.$$
As above, on one hand since $T$ appears as an operator of rank one, we have
$$\|t\|_{\min}\ge   \|T\|_{B({\ell^{N(1)}_2} \otimes_2\cdots  \otimes_2  {\ell^{N(d)}_2})  }\ge (\sum  |g_{i(1),...,i(d)}|^2)^{1/2}  (\sum  |g'_{i(1),...,i(d)}|^2)^{1/2}  .$$
On the other hand using the  orthogonality of the $u_{\underline{ i}} $'s we have
$$\|t\|_{\vee}\le  {(N(1)\cdots N(d))}^{1/2} \|T\|_{\ell^{N(1)^2}_2 {\buildrel {\vee}\over {\otimes}}\cdots  
{\buildrel {\vee}\over {\otimes}} \ell^{N(d)^2}_2}. $$ 
Thus we find
$$C_d({N(1)^2}, \cdots,{N(d)^2})\geq     {(N(1)\cdots N(d))}^{-1/2}  \left(\sup \|T\|_{\ell^{N(1)^2}_2 {\buildrel {\vee}\over {\otimes}}\cdots  
{\buildrel {\vee}\over {\otimes}} \ell^{N(d)^2}_2}\right)^{-1},$$
where the sup runs over all $T$ of the above form (i.e. of rank one in a suitable sense)
such that $(\sum  |g_{i(1),...,i(d)}|^2)^{1/2}  (\sum  |g'_{i(1),...,i(d)}|^2)^{1/2} \le 1$
(i.e. of norm one in a suitable sense).
Then using Gaussian variables as above, we obtain 
$  \|t\|_{\min}\ge c N^d$ and $  \|t\|_{\vee}\le c N^{d/2} (N (\log N)^{d/2})$
$$C_d({N^2}, \cdots,{N^2})\geq   c_d N^{d/2-1} (\log N)^{-d/2}.$$
\end{rem}
\section{An almost sharp inequality}
Let $(H_j)$ and $(K_j)$  ($1\le j\le d$) be $d$-tuples  of finite dimensional Hilbert spaces.
Let
$$ J:\ (H_1\otimes_2 K_1)\buildrel{\vee}\over{\otimes} \cdots\buildrel{\vee}\over{\otimes} (H_d\otimes_2 K_d) \to (H_1  \otimes_2\cdots \otimes_2 H_d) \buildrel{\vee}\over{\otimes}  (K_1  \otimes_2\cdots \otimes_2 K_d), $$
be the natural identification map. After reordering, we may as well assume that
the sequence $ \{ \dim(H_j)\dim(K_j)\mid 1\le j\le d\} $  is non-decreasing.

We have
\begin{equation}\label{eq0}\|J\|\le \prod_{j=1}^{d-1} (\dim(H_j)\dim(K_j))^{1/2} .\end{equation}
Indeed,  it is easy to check that, for any normed space $E$, the identity map
$\ell_2^n \buildrel{\vee}\over{\otimes} E \to \ell_2^n(E)$ has norm $\le \sqrt{n}$.
This gives
$$\| (H_1\otimes_2 K_1)\buildrel{\vee}\over{\otimes} \cdots\buildrel{\vee}\over{\otimes} (H_d\otimes_2 K_d) 
\to (H_1\otimes_2 K_1)\otimes_2   \cdots\otimes_2(H_d\otimes_2 K_d) \|\le \prod_{j=1}^{d-1} (\dim(H_j)\dim(K_j) )^{1/2}.$$
A fortiori we obtain the above bound \eqref{eq0} for $J$.

Consider now the case when $\dim(H_j)=\dim(K_j)=N $ for all $j$. In that case the preceding bound becomes
$$\|J\|\le N^{d-1}.$$
Consider again
$$T=\sum g_{i(1),...,i(d)}   g'_{i'(1),...,i'(d)}   e_{i(1)i'(1)} \otimes\cdots  \otimes e_{i(d)i'(d)}\in \ell^{N^2}_2 \otimes\cdots  \otimes \ell^{N^2}_2,$$
where we identify $H_j\otimes_2 K_j$ with $ \ell^{N^2}_2$. The preceding proof
yields with large probability
$$\|T\|_{\ell^{N^2}_2 \buildrel{\vee}\over{\otimes} \cdots\buildrel{\vee}\over{\otimes} \ell^{N^2}_2}\le c  N (\log N)^{d/2}$$
and
also  
$$\|T\|_{ (H_1  \otimes_2\cdots \otimes_2 H_d) \buildrel{\vee}\over{\otimes}  (K_1  \otimes_2\cdots \otimes_2 K_d)}\ge c ' N  ^{d}.$$Thus we obtain the following almost sharp (i.e. sharp up to the log factor)
$$\|J\|\ge c''  N^{d-1}(\log N)^{-d/2}.$$
The argument described in the preceding Remark \ref{rem7} boils down to the estimate
$$  C_d(N^2,\cdots ,N^2)  \ge N^{-d/2}\|J\| .$$

\begin{rem}Note that the preceding proof actually yields  (assuming $\dim(H_j)=\dim(K_j)=N $ for all $j$)
$$ c''  N^{d-1}(\log N)^{-d/2}\le \| (H_1\otimes_2 K_1)\buildrel{\vee}\over{\otimes} \cdots\buildrel{\vee}\over{\otimes} (H_d\otimes_2 K_d) 
\to (H_1\otimes_2 K_1)\otimes_2   \cdots\otimes_2(H_d\otimes_2 K_d) \|\le N^{d-1}.$$
However, this norm is much easier to estimate and,  actually, we claim  it is  $\ge  c_d N^{d-1}.$\\
Indeed, returning to $H_j,K_j$ of arbitrary finite dimension,   let $n_j=\dim(H_j)\dim(K_j)$.\\
Assume $n_1\le n_2\le \cdots\le n_d$. Consider then the inclusion
$$\Phi:\ \ell_2^{n_1}\buildrel{\vee}\over{\otimes} \cdots \buildrel{\vee}\over{\otimes} \ell_2^{n_d} \to \ell_2^{n_1\cdots n_d}.$$ 
The above easy argument shows that $\|\Phi\|\le ({n_1\cdots n_{d-1}})^{1/2}$.
Let now $G$ be a   random vector with values in $\ell_2^{n_1\cdots n_d}$ distributed according
to the canonical Gaussian measure. We will identify $G$ with $\Phi^{-1}(G)$. Then, by Simone Chevet's well known inequality
we have
$$\E \|G\|_{\ell_2^{n_1}\buildrel{\vee}\over{\otimes} \cdots \buildrel{\vee}\over{\otimes} \ell_2^{n_d}}
\le \sqrt{d} \sum_j \sqrt{n_j}\le d^{3/2} n_d,$$
while it is clear that $\E \|G\|^2_{ \ell_2^{n_1\cdots n_d}}= {n_1\cdots n_d}$. From this follows
that $\|\Phi\|\ge d^{-3/2} ({n_1\cdots n_{d-1}})^{1/2}$. In particular the above claim is established.
\end{rem}
\section{A different method}
It is known (due to Geman) that
\begin{equation}\label{eq13.4}
 \lim_{N\to\infty} \|Y^{(N)}\|_{M_N} = 2 \quad \text{a.s.}
\end{equation}
Let $(Y^{(N)}_1, Y^{(N)}_2,\ldots)$ be a sequence of independent copies of $Y^{(N)}$, so that the family $\{Y^N_k(i,j)\mid k\ge 1, 1\le i,j\le N\}$ is an independent family of $N(0,N^{-1})$ complex Gaussian.

The next two statements follow from results  known to Steen Thorbj{\o}rnsen since at least 1999 (private communication). See \cite{HT3} for closely related results. We present a trick that yields a  self-contained derivation of this.  
  
\begin{thm}\label{thm13.10}
Consider independent copies $Y'_i = Y^{(N)}_i(\omega')$ and
$
 Y''_j = Y^{(N)}_j(\omega'')$    {for} $ (\omega',\omega'')\in\Omega\times \Omega$.
 Then, for any $n^2$-tuple of scalars $(\alpha_{ij})$, we have
\begin{equation}\label{eq13.16bis}
 \underset{N\to\infty}{\ovl{\lim}}\left\|\sum \alpha_{ij}Y^{(N)}_i(\omega') \otimes Y^{(N)}_j(\omega'')\right\|_{M_{N^2}} \le 4(\sum|\alpha_{ij}|^2)^{1/2}
\end{equation}
for a.e.\ $(\omega',\omega'')$ in $\Omega\times \Omega$.\end{thm}

\begin{proof}
By (well known) concentration of measure arguments, it is known that \eqref{eq13.4} is essentially the same as the assertion that $\lim_{N\to\infty} {\bb E}\|Y^{(N)}\|_{M_N}=2$. Let $\vp(N)$ be defined by
\[
 {\bb E}\|Y^{(N)}\|_{M_N} = 2+\vp(N)
\]
so that we know $\vp(N)\to 0$.
Again by concentration of measure arguments (see e.g.\ \cite[p.~41]{Led} or \cite[(1.4) or chapter 2]{P02}) there is a constant $\beta$ such that for any $N\ge 1$ and  $p\ge 2$ we have 
\begin{equation}\label{eqg1}
({\bb E}\|Y^{(N)}\|^p_{M_N})^{1/p} \le {\bb E}\|Y^{(N)}\|_{M_N} + \beta(p/N)^{1/2} \le 2+\vp(N) + \beta(p/N)^{1/2}.
\end{equation}
For any $\alpha\in M_n$, we denote
\[
 Z^{(N)}(\alpha)(\omega',\omega'') = \sum\nolimits^n_{i,j=1} \alpha_{ij} Y^{(N)}_i(\omega') \otimes Y^{(N)}_j(\omega'').
\]
Assume $\sum_{ij}|\alpha_{ij}|^2 = 1$. We will show that almost surely
\[
 \lim\nolimits_{N\to\infty} \|Z^{(N)}(\alpha)\| \le 4.
\]
Note that by the invariance of (complex) canonical Gaussian measures under unitary transformations, $Z^{(N)}(\alpha)$ has the same distribution as $Z^{(N)}(u\alpha v)$ for any pair $u,v$ of $n\times n$ unitary matrices. Therefore, if $\lambda_1,\ldots, \lambda_n$ are the eigenvalues of $|\alpha| = (\alpha^*\alpha)^{1/2}$, we have
\[
 Z^{(N)}(\alpha)(\omega',\omega'') \overset{\text{dist}}{=} \sum\nolimits^n_{j=1} \lambda_j Y^{(N)}_j(\omega') \otimes Y^{(N)}_j(\omega'').
\]
We claim that by a rather simple calculation of moments, one can show that for any even integer $p\ge 2$ we have
\begin{equation}\label{eqg2}
 {\bb E} \text{ tr}|Z^{(N)}(\alpha)|^p \le ({\bb E} \text{ tr}|Y^{(N)}|^p)^2.
\end{equation}
Accepting this claim for the moment, we find, a fortiori,  using \eqref{eqg1}:
\[
 {\bb E}\|Z^{(N)}(\alpha)\|^p_{M_N} \le N^2({\bb E}\|Y^{(N)}\|^p_{M_N})^2 \le N^2(2+\vp(N)+ \beta(p/N)^{1/2})^{2p}.
\]
Therefore for any $\delta>0$
\[
 {\bb P}\{\|Z^{(N)}(\alpha)\|_{M_N} > (1+\delta)4\} \le (1+\delta)^{-p} N^2(1+\vp(N)/2 + (\beta/2)(p/N)^{1/2})^{2p}.
\]
Then choosing (say) $p=5(1/\delta) \log( N)$ we find
\[
 {\bb P}\{\|Z^{(N)}(\alpha)\|_{M_N} > (1+\delta)4\} \in O(N^{-2})
\]
and hence (Borel--Cantelli) $\ovl{\lim}_{N\to\infty} \|Z^{(N)}(\alpha)\|_{M_N} \le 4$ a.s.. \\
It remains to verify the claim. Let $Z=Z^N(\alpha)$, $Y=Y^{(N)}$ and $p=2m$. We have
\[
 {\bb E} \text{ tr}|Z|^p = {\bb E} \text{ tr}(Z^*Z)^m = \sum \bar\lambda_{i_1}\lambda_{j_1} \ldots \bar\lambda_{i_m}\lambda_{j_m}({\bb E} \text{ tr}(Y^*_{i_1}Y_{j_1}\ldots Y^*_{i_m}Y_{j_m}))^2.
\]
Note that the only nonvanishing terms in this sum correspond to certain pairings that guarantee that both $\bar\lambda_{i_1}\lambda_{j_1}\ldots \bar\lambda_{i_m}\lambda_{j_m}\ge 0$ and ${\bb E} \text{ tr}(Y^*_{i_1}Y_{j_1}\ldots Y^*_{i_m}Y_{j_m})\ge 0$. Moreover, by H\"older's inequality for the trace we have
\[
 |{\bb E} \text{ tr}(Y^*_{i_1}Y_{j_1}\ldots Y^*_{i_m}Y_{j_m})| \le \Pi({\bb E} \text{ tr}|Y_{i_k}|^p)^{1/p} \Pi({\bb E} \text{ tr}|Y_{j_k}|^p)^{1/p} = {\bb E} \text{ tr}(|Y|^p).
\]
From these observations, we find
\begin{equation}\label{eqg3}
 {\bb E} \text{ tr}|Z|^p \le {\bb E} \text{ tr}(|Y|^p) \sum \bar\lambda_{i_1}\lambda_{j_1}\ldots \bar\lambda_{i_m}\lambda_{j_m} {\bb E} \text{ tr}(Y^*_{i_1}Y_{j_1}\ldots Y^*_{i_m} Y_{j_m})
\end{equation}
but the last sum is equal to ${\bb E} \text{ tr}(|\sum \lambda_jY_j|^p)$ and since $\sum \lambda_jY_j \overset{\text{dist}}{=} Y$   (recall $\sum |\lambda_j|^2 = \sum|\alpha_{ij}|^2=1$) we have
\[
 {\bb E} \text{ tr}\Big(\Big|\sum \alpha_jY_j\Big|^p\Big) = {\bb E} \text{ tr}(|Y|^p),
\]
and hence \eqref{eqg3} implies \eqref{eqg2}.
\end{proof}

\begin{cor}\label{cor13.11}
For any integer $n$ and $\vp>0$, there are $N$ and $n$-tuples of $N\times N$ matrices $\{Y'_i\mid 1\le i\le n\}$ and $\{Y''_j\mid 1\le j\le n\}$ in $M_N$ such that
\begin{align}\label{eq13.15}
 &\sup\left\{\left\|\sum^n_{i,j=1} \alpha_{ij}Y'_i \otimes Y''_j\right\|_{M_{N^2}}\ \Big| \ \alpha_{ij}\in {\bb C}, \ \sum\nolimits_{ij}|\alpha_{ij}|^2 \le 1\right\} \le (4+\vp)\\
\label{eq13.16}
&\min\left\{\frac1{nN} \sum\nolimits^n_1 \text{ \rm tr}|Y'_i|^2, \frac1{nN} \sum\nolimits^n_1 \text{ \rm tr}|Y''_j|^2\right\} \ge 1-\vp. 
\end{align}
\end{cor}

\begin{proof}

Fix $\vp>0$. Let ${\cl N}_\vp$ be a finite $\vp$-net in the unit ball of $\ell^{n^2}_2$. By 
Theorem \ref{thm13.10} we have for almost all $(\omega',\omega'')$  
\begin{equation}\label{eq13.18bis}
   \ovl{\lim}_{N\to\infty} \sup_{\alpha\in {\cl N}_\vp}\left\|\sum\nolimits^n_{i,j=1} \alpha_{ij} Y'_i\otimes Y''_j\right\|_{M_{N^2}} \le 4,
\end{equation}
We may pass from an $\vp$-net to the whole unit ball in \eqref{eq13.18bis} at the cost of an extra factor $(1+\vp)$ and we obtain \eqref{eq13.15}. As for \eqref{eq13.16}, the strong law of large numbers shows that the left side of \eqref{eq13.16} tends a.s.\ to 1. Therefore, we may clearly  find $(\omega',\omega'')$ satisfying both \eqref{eq13.15} and \eqref{eq13.16}.
\end{proof}
\begin{rem}\label{rem2} A close examination of the proof 
and concentration of measure arguments show
that the preceding corollary holds
with $N$ of the order of $c(\vp) n^2$.
Indeed, we find
a constant $C$ such that
for any $\alpha=(\alpha_{ij})$ in the unit ball
of $\ell_2^{n^2}$ we have (we take $p=N$)
$$ \|Z^{(N)}(\alpha)\|_{L_N(M_N)} \le C $$
from which follows if $A$ is a finite subset of the 
unit ball
of $\ell_2^{n^2}$ that
$$ \|\sup_{\alpha\in A}\|Z^{(N)}(\alpha)\|_{ M_N}\|_N \le C |A|^{1/N}. $$
So if we choose for $A$ an $\vp$-net
in the unit ball
of $\ell_2^{ n^2}$ with $|A|\approx 2^{cn^2}$,
and if $N=n^2$
we still obtain a fortiori

$$ \|\sup_{\alpha\in A}\|Z^{(N)}(\alpha)\|_{ M_N}\|_1 \le C'. $$

\end{rem}
\begin{rem}\label{rem13.12}
Using the well known ``contraction principle'' that says that the variables $(\vp_j)$ are dominated by either $(g^{\bb R}_j)$ or $(g^{\bb C}_j)$, it is easy to deduce that Corollary~\ref{cor13.11} is valid for matrices $Y'_i,Y''_j$ with entries all equal to $\pm N^{-1/2}$, with possibly a different numerical constant in place of 4. Analogously, using the polar factorizations  $Y'_i=U'_i|Y'_i|$, $Y''_j=U''_j|Y''_j |$ and noting
that all the factors $U'_i,|Y'_i| ,U''_j,|Y''_j |$ are independent, we
can also (roughly by integrating over the moduli $|Y'_i| , |Y''_j |$) obtain Corollary~\ref{cor13.11} with unitary matrices $Y'_i,Y''_j$ , with a different numerical constant in place of 4. 
\end{rem}

Let $C_3(n_1,n_2,n_3)$ the  supremum appearing
in \eqref{12} when $t$ runs over all tensors
in $\ell_1^{n_1} \otimes \ell_1 ^{n_2}\otimes \ell_1 ^{n_3} $.  Note that
$C_3(n)= C_3(n,n,n)  $.  Then the proof of Junge etal
as presented in \cite{P-gro} (and incorporating
the results of \cite{P7}) yields
$$C_3(n^4,n^8,n^8)\ge c n^{1/2}.$$
Indeed, the latter proof requires an embedding of $\ell_2^{n^2}$ into $\ell_1^{m}$
and Junge etal use the Rademacher embedding, and hence $m=2^{n^2}$,
but \cite{P7} allows us to use $m=n^4$.

However, if we use the method of Theorem \ref{thm1}    together with Corollary \ref{cor13.11} and the estimate
in Remark \ref{rem2}, then we find
$$C_3(n^2,n^4,n^4)\ge c n^{1/2}.$$

The open problems that remain are:

--get rid of the log factor in Theorem \ref{thm1}.

--improve the lower bound of $C_3(n)$ in Theorem \ref{thm2} to something sharp, possibly $cn^{1/2}$, and similar questions
for $C_3(n_1,n_2,n_3)$.

--find explicit non random examples responsible for 
large values of $C_3(n)$.
\section{Upper bounds }
As far as I know the upperbounds for $C_3(n)$ or  $C_d(n)$ are as follows.

First we have $C_2(n)\le K_G$ (here $K_G$ is the Grothendieck constant).

If $E$ is any operator space,  let $\min(E)$ be the same Banach space but viewed
as embedded in a commutative $C^*$-algebra (i.e. the continuous functions on the dual unit ball).

Let $F$ be an arbitrary operator space, it is easy to show that  we have   isometric  identities
$$F \otimes_{\min} \min(E) =F  \buildrel{\vee} \over {\otimes} E=\min(F) \otimes_{\min} \min(E)   .$$
Moreover, it is known (and easy to check) that if $E=\ell_1^n   $,
the identity map $\min(E) \to E$
has cb norm at most $\sqrt{n}$.  A fortiori, we have
$\|F \otimes_{\min} \min(E)\to F \otimes_{\min} E  \|_{cb} \le \sqrt{n}$ and hence
$$\| F  \buildrel{\vee} \over {\otimes} E \to F \otimes_{\min} E  \|\le \sqrt{n} .$$
This implies that if $E_d= \ell_1^n \otimes_{\min} \cdots \otimes_{\min}  \ell_1^n $ ($d$ times),
then
$$\|E_{d-1}  \buildrel{\vee} \over {\otimes} E \to E_{d}  \|\le \sqrt{n}.$$
Iterating we find
$$ \| E \buildrel{\vee} \over {\otimes}  \cdots \buildrel{\vee} \over {\otimes}  E\to E_d\|\le C_{d-1}(n) \sqrt{n}.$$
Thus we obtain
$$C_d(n)\le K_G n^{(d-2)/2}.$$
A similar argument yields (note that we can use invariance under permutation to reduce to the case
when $n_1\le n_2\le \cdots$) whenever $d\ge 3$:
$$C_d(n_1,\cdots,n_d)\le  K_G \left( {n_1\cdots n_{d-2}}  \right)^{1/2}.$$

 \end{document}